\documentclass{amsart}
\usepackage{amsmath,amssymb,amsthm,amsfonts,amscd,mathrsfs}
\usepackage{tikz}
\usepackage{braids}
\usepackage[all]{xy}
\usepackage{url}

\def\co{\colon\thinspace}

\theoremstyle{plain}
    \newtheorem{thm}{Theorem}[section]
    \newtheorem{lem}[thm]   {Lemma}

\theoremstyle{definition}
    \newtheorem{defn}[thm]  {Definition}
    
    \newtheorem{qns}[thm] {Questions}

    \newtheorem{rem}[thm]{Remark}

\begin{document}

\title{Homologies are infinitely complex}

\author{Mark Grant}
\author{Andr\'{a}s Sz\H{u}cs}

\address{School of Mathematics \& Statistics,
Herschel Building,
Newcastle University,
Newcastle upon Tyne NE1 7RU,
UK}

\email{mark.grant@newcastle.ac.uk}

\address{E\" otv\" os Lor\' and University,
P\' azm\' any P\' eter s\' etany 1/C,
 3-206, 1117 Budapest,
Hungary}

\email{szucs@math.elte.hu}



\begin{abstract} We show that for any $k>1$, stratified sets of finite complexity are insufficient to realize all homology classes of codimension $k$ in all smooth manifolds. We also prove a similar result concerning smooth generic maps whose double-point sets are co-oriented. \end{abstract}

\thanks{The second author was supported in part by Grant OTKA NK 81203.}

\maketitle

\section{Introduction}

Realizing homology classes by manifolds is a classical problem in Algebraic Topology.

\begin{enumerate}
\item It is well known that any $\mathbb{Z}_2$-homology class of any space $X$ can be realized by a manifold in the sense that for any $x \in H_{m}(X; \mathbb{Z}_2)$ there is an $m$-dimensional closed, smooth manifold and a map $f\co M^m \to X$ such that $x = f_*[M]$ (see \cite{4}).

\item When $X$ itself is a manifold, then one can ask, how nice the map $f$ can be chosen.
Thom gave a necessary and sufficient condition to have an embedding realizing the class~$x$.
In particular, if $\dim X - \dim x = 1$, then the homology class $x$ can be realized by a smooth embedding.

We shall say that a cohomology class $\alpha \in H^k(P^p; \mathbb{Z}_2)$ of a smooth manifold $P^p$ is realized by a map $f\co\ M^m \to P^p$ if $\alpha$ is Poincar\'e dual to $f_*[M^m]$ $(p = m + k)$.

\item In \cite{1} it was shown that for any $k > 1$ there are $k$-dimensional cohomology classes of some sufficiently high-dimensional manifolds which cannot be realized by immersions.

\item In \cite{1} we have shown also that if we allow only a finite set of multisingularities, then not all cohomology classes can be realized.
\end{enumerate}

In the present paper we show that stratified sets of finite complexity are never sufficient to realize all homology classes of codimension $k > 1$ (Theorem \ref{A}).
The proof is almost identical to that of Theorem 1.3 in \cite{1}.

  Further we show that smooth maps with any types of local singularities are not sufficient to realize all homology classes if we require the set of (regular) double points to be co-oriented (Theorem \ref{B}).

\section{Stratified sets of finite complexity}

\begin{defn}\label{2.1}
Let $\theta$ be a finite set with elements $\eta_1, \dots, \eta_r$, where each $\eta_i$ is a triple: $\eta_i = (A_i, D^{c_i}, G_i)$. Here:

\begin{enumerate}
\item $c_i$ is a natural number and $D^{c_i}$ is the $c_i$-dimensional unit ball in $\mathbb R^{c_i}$.
\item $G_i$ is a compact subgroup of $O(c_i)$.
\item $A_i$ is a $G_i$-invariant subset of $D^{c_i}$.
We shall suppose that $A_i$ is a cone over a simplicial subcomplex of the sphere $S^{c_i - 1} = \partial D^{c_i}$.
\item We shall suppose also that there is a $k$ such that each $A_i$ has codimension $k$ in $D^{c_i}$ for $i = 1,\dots, r$.
\end{enumerate}
\end{defn}

\begin{defn}\label{2.2}
A subset $K$ of a smooth manifold $P^p$ is called a {\it $\theta$-subset} if the following hold:
\begin{enumerate}
\item For each $x \in K \subset P$ there are an $i$, $1 \leq i \leq r$, a neighbourhood $U$ of $x$, and a diffeomorphism $\varphi_U$ of $U$ onto the product $D^{c_i} \times D^{p - c_i}$ such that the image of $K \cap U$ is $A_i \times D^{p - c_i}$.
Such a point $x$ will be called a {\it point of type} $\eta_i$ in $K$.
We shall say also that such a point $x$ has {\it normal structure type} $\eta_i$.
\item For any given $i$ $(1 \leq i \leq r)$ the set of points of the type $\eta_i$ form a smooth submanifold $\eta_i(K)$ in $P$.
Let $T_i$ be the tubular neighbourhood of $\eta_i(K)$ in $P$.
\item The $D^{c_i}$-bundle $T_i \to \eta_i(K)$ admits $G_i$ as a structure group.
More precisely the maps $\varphi_U$ define a $G_i$-structure on it in the sense that for any $x \in \eta_i(K)$ and $U$, $V$ neighbourhoods as in 1) with maps $\varphi_U$ and $\varphi_V$ the map
$\varphi_V \circ \varphi_U^{-1}$ restricted to any ball $D^{c_i} \times q$ with $q \in D^{p - c_i}$ belongs to $G_i$ (if it is defined).
\end{enumerate}
\end{defn}

\begin{rem}
Any $\theta$-subset in a smooth manifold is a codimension $k$ stratified subset
(by (4) in Definition \ref{2.1}),
hence the highest dimensional stratum has codimension~$k$.
If the second highest dimensional stratum has codimension at least $k + 2$, then a $\theta$-subset is a cycle and so it represents a codimension $k$ homology class of the ambient smooth manifold.
We shall say also that the $\theta$-subset represents the dual $k$-dimensional cohomology class.
\end{rem}

\begin{thm} \label{A}
Let $k$ be greater than $1$ and let $\theta$ be a finite set (of normal structures) as above, so that $\theta$-subsets in smooth manifolds represent $k$-dimensional cohomology classes.
Then there is a smooth manifold $P$ (of sufficiently high dimension) and a $k$-dimensional cohomology class of $P$ that can not be represented by a $\theta$-subset.
\end{thm}

In other words, stratified subsets of codimension $k$ with finitely many allowed normal structure types are never enough to realize all $k$-dimensional cohomology classes.

\section{Preliminaries for the proof: The Pontryagin--Thom construction for $\theta$-subsets}

\begin{defn}
Two $\theta$-subsets $K_0$ and $K_1$ in a manifold $P$ are said to be {\it cobordant} if there is a $\theta$-subset $W$ in $P \times I$ such that $W \cap (P \times \{\varepsilon\})$ is $K_0 \times \{\varepsilon\}$ if $0 \leq \varepsilon \leq \frac13$ and it is $K_1 \times \{\varepsilon\}$ if $\frac23 \leq \varepsilon \leq 1$.
The set of $\theta$-cobordism classes in $P$ will be denoted by $\theta(P)$.
\end{defn}

\begin{rem}
If $f \co N \to P$ is any continuous map of the smooth manifold $N$ into the smooth manifold $P$, then it induces a map $f^*\co\ \theta(P) \to \theta(N)$ in the following way:

Given a $\theta$-subset $K$ in $P$ make the map $f$ transverse to each stratum of $K$ and take the inverse images of these strata.
The tubular neighbourhoods in $N$ of the preimages of the strata of $K$ will have $G_i$-structures, hence in these neighbourhoods the normal structures $\eta_i$ arise, $i = 1,2, \dots, r$.
\end{rem}

\begin{rem}
If $P$ has the form $P'\times \mathbb R^1$ for some manifold $P'$, then $\theta(P)$ will be a group.
(The inverse element is obtained by applying the map $(q,t) \longmapsto (q, -t)$ for $q \in P'$, $t \in \mathbb R^1$.)

If the source manifold $N$ also has the form $N'\times \mathbb R^1$ and $f$ has the form $f'\times \text{\rm id}_{\mathbb R^1}$ where $f'$ is a map $N' \to P'$, then $f^*$ is a homomorphism.
\end{rem}

Notation:
\begin{enumerate}
\item If $Y$ is any space, then we shall denote by $Y_{\infty}$ its one-point
compactification. In particular if $Y$ is compact, then $Y_{\infty}$ is the
disjoint union of $Y$ with a point. So $Y_{\infty}$ is a pointed space (i.e.
a space with a marked point).

\item If $V$ is any pointed space, then $[Y,V]$ will denote the set of homotopy
classes of pointed maps $Y_{\infty} \to V$ (i.e. continuous maps mapping the marked
point to the marked point, and the homotopies must go through pointed maps.)
\end{enumerate}

\begin{thm}\label{PT}
There is a Pontryagin--Thom construction for cobordisms of $\theta$-subsets.
That is, there is a classifying space $X_\theta$ such that $\theta(P)$ can be
canonically identified with the set of homotopy classes $[P, X_\theta]$ for
any closed manifold $P$. (Note, that $X_\theta$ is a pointed space, we shall
see that below.)
Moreover, for any map $f\co\ N \to P$ the induced map $f^*\co \ \theta(P) \to \theta(N)$ will be identified with the map
$[P ,  X_\theta] \to [N , X_\theta]$, induced by taking the precomposition
with the map $f_\infty\co N_{\infty} \to P_{\infty}$ (induced by $f$),
$[\alpha] \longmapsto [\alpha \circ f_{\infty}]$, where $\alpha$ is a pointed map $\alpha \co \ P_{\infty} \to X_\theta$.
\end{thm}

The construction of $X_\theta$ goes by induction.
The starting of the induction:
When $\theta$ consists of a single element $\eta_1$, then necessarily $A_1$ is the centre of $D^{c_1}$, and $\theta$-subsets are simply embedded submanifolds of codimension $k$ with normal bundle reduced to the structure group $G_1$.
Let $MG_1$ be the Thom space of the corresponding universal vector bundle.
Then $X_\theta = MG_1$ by the classical Pontryagin-Thom construction.
(The Thom space $MG_1$ is a pointed space, and its marked point will be the
marked point of $X_{\theta}$, for any $\theta$. From the induction step it will
  be clear, that $MG_1$ is a subspace of $X_{\theta}$ for any $\theta$.)

The induction step:
We can suppose that the elements $\eta_1, \dots, \eta_r$ of the set $\theta$ are enumerated in such a way that
$\dim \eta_i(K) \leq \dim \eta_{i + 1}(K)$ for any $\theta$-subset~$K$ and $i = 1,2, \dots, r - 1$.

Let $\theta'$ be the subset of $\theta$ obtained by omitting $\eta_r$.
We shall denote $\eta_r = (A_r, D^{c_r}, G_r)$ simply by $\eta = (A, D^c, G)$.
Let $X'$ be the classifying space for $\theta'$-subsets.
Let us denote by $\xi$ the bundle $EG {}_{G}{\times} \mathbb{R}^c$, where $EG$ is a contractible, free $G$-space.
Note that on the boundary of $D(\xi)$
(the associated ball-bundle) a $\theta'$-subset arises:
$K' = EG \times_G (A \cap S^{c - 1})$.
(More precisely for any finite dimensional manifold-approximation $(D\xi)_{\text{finite}}$ of $D\xi$ the set $\partial (D\xi)_{\text{finite}}\cap K'$ is a $\theta'$-subset of $\partial (D\xi)_{\text{finite}}$.)
Hence by the hypothesis of the induction a map
$\varrho \co  \partial D(\xi) \to X'$ arises.
Now we define $X_\theta$ as follows: $X_\theta = X' \bigcup_\varrho D(\xi)$
(i.e. attach $D(\xi)$ to $X'$ using the map $\varrho$).
It is standard to show that $\theta(P) = [P, X_\theta]$. (See the analogous result for singular maps in \cite{2}.)

\section{Proof of Theorem \ref{A}}

Let $U$ be the Thom class of the Thom space $MG_1$.
By construction the space $MG_1$ is a subset of $X_\theta$, and the inclusion
$j \co \ MG_1 \subset X_\theta$ induces an isomorphism of the $k$-th cohomology
groups: $j^*\co\ H^k(X_\theta; \mathbb{Z}_2) \widetilde{\longrightarrow} H^k
(MG_1, \mathbb{Z}_2)$ ($X_\theta$ is obtained from $MG_1$ by successive cofibrations,
the cofibers of which are the Thom spaces of bundles of dimensions at least $(k + 2)$, hence the $k$-th cohomology group does not change).
We shall denote also by $U$ the element $(j^*)^{-1}(U)$ in $H^k(X_\theta; \mathbb{Z}_2)$.
Considering $U$ as a map $U\co X_\theta \to K(\mathbb{Z}_2, k)$, its composition with a map $f\co\ P \to X_\theta$ induces the map $\theta(P) \to H^k(P; \mathbb{Z}_2)$ that associates to (the cobordism class of) a $\theta$-subset the represented $k$-dimensional cohomology class.

If a $k$-dimensional cohomology class $y\co\ P \to K(\mathbb{Z}_2, k)$ can be represented by a $\theta$-subset, then there is a lift $\widetilde y \co\ P \to X_\theta$ such that $U\circ \widetilde y = y$.

If any $k$-dimensional cohomology class of any manifold $P$ can be represented by a $\theta$-subset, then such a lift $\widetilde y$ exists for any $P$ and any map $y\co\ P \to K(\mathbb{Z}_2, k)$, so the identity map $\text{\rm id}\co\ K(\mathbb{Z}_2, k) \to K(\mathbb{Z}_2, k)$ also has a lift $\widetilde{\text{\rm id}}\co\ K(\mathbb{Z}_2, k) \to X_\theta$ such that $U \circ \widetilde{\text{\rm id}} = \text{\rm id}$.
Hence $U^* \co\ H^*\bigl(K(\mathbb{Z}_2, k); \mathbb{Z}_2\bigr) \to H^*(X_\theta; \mathbb{Z}_2)$ is an injection.

On the other hand, by a classical result of Serre \cite{3} the ring $H^*\bigl(K(\mathbb{Z}_2, k); \mathbb{Z}_2\bigr)$ is a polynomial algebra with infinitely many variables.
It is well known that the dimension of the degree $n$ part of such an algebra grows faster (when $n \to \infty$) than the degree $n$ part of any polynomial algebra with finitely many variables (or of any finite sum of such algebras).

Hence Theorem \ref{A} will be proved if we show the following

\begin{lem}
The sequence $\dim H^n(X_\theta; \mathbb{Z}_2)$ grows not faster (when $n \to \infty$) than the dimension of the degree $n$ part of a finite sum
of finitely generated polynomial algebras.
\end{lem}

\begin{proof}
By construction $X_\theta$ has a filtration:
$$
X_0 \subset X_1 \subset \dots \subset X_r = X_\theta,
$$
where $X_0$ is a point and $X_i / X_{i - 1}$ is the Thom space of a vector bundle (of $\dim c_i$) over $BG_i$.
Hence
$$
\dim H^n(X_\theta; \mathbb{Z}_2) \leq \sum_{i = 1}^r \dim H^{n - c_i} (BG_i; \mathbb{Z}_2).
$$
By a result of Venkov \cite{5}, for any compact Lie group $G$ the sequence $\dim H^n(BG; \mathbb{Z}_2)$ grows not faster than the dimension of the degree $n$ part of a polynomial algebra with finitely many variables.
\end{proof}

This concludes the proof of Theorem \ref{A}. \qed

\medskip
It is natural to pose the following questions (We do not know the answers):

\begin{qns}
\begin{enumerate}
\item Is it possible to realize any homology class of any smooth manifold by maps having {\it local} singularities from a given finite set?
(For example: Is it possible to realize any homology class in smooth manifolds by fold maps?)

\item Is it possible to realize all homology classes by corank $i$ maps for a fixed $i$?
\end{enumerate}
\end{qns}

\section{Maps with arbitrary local singularities and co-oriented double points}

 Here we show that not all homology classes in smooth manifolds can be realized by smooth maps even if we allow any local singularities but require the second top dimensional stratum to be co-oriented.

\begin{defn}
Let a smooth generic map $f\co\ N^n \to P^p$ be called {\it double-points co-oriented}, or a {\it dc} map for short, if the push-forward of the virtual normal bundle $f_!(\nu_f)$ becomes oriented when we restrict it to the set of the image of double points.
\end{defn}

\begin{rem}\label{5.2}
Genericity of the map will imply that the set of points of $P$ having exactly two non-singular preimage points form a $(p - 2k)$-dimensional chain
 (when properly triangulated) with boundary equal to the image of the fold-singular points.
\end{rem}

\begin{thm} \label{B}
For any even $k$ there is a cohomology class of dimension $k$ in a smooth manifold of sufficiently high dimension which can not be realized by any dc-map.
\end{thm}

\begin{proof}
In \cite[Proposition 4.2]{1} it was shown that if $f\co\ M^m \to P^p$ is a generic smooth map, $p = m + k$, and $\alpha \in H^k(P; \mathbb{Z}_2)$ the cohomology class realized by $f$, then the class $\beta(\alpha^2)\in H^{2k+1} H(P; \mathbb{Z})$ (where $\beta$ is the Bockstein operator) is dual to the homology class $f_*\bigl[\Sigma(f)\bigr]$, where $\Sigma(f)$ is the set of singular points.

By Remark \ref{5.2}, for a dc map the class $\beta(\alpha^2)$ must be zero.
(The cycle $f\bigl(\Sigma(f)\bigr)$ is the boundary of the chain formed by the regular double points in the chain group $C_{p - 2k - 1}(P; \mathbb{Z})$.)

On the other hand, it was shown in \cite[Theorem 1.1]{1} that there is a manifold $P$ and a class $\alpha \in H^{k}(P; \mathbb{Z}_2)$ for which $\beta(\alpha^2)$ is not zero.
\end{proof}

\end{document}